\documentclass[a4paper, reqno, 12pt]{amsart}

\usepackage[utf8]{inputenc}
\usepackage[T1]{fontenc}
\usepackage[english]{babel}

\usepackage{amsmath}
\usepackage{amssymb}
\usepackage{amsthm}
\usepackage{mathtools}
\usepackage{booktabs}
\usepackage[mathscr]{euscript}
\usepackage{graphicx}
\usepackage{xparse}
\usepackage{stmaryrd}
\usepackage{csquotes}
\usepackage[backend=biber]{biblatex}

\NewDocumentCommand \virg { +m } {``#1''}

\NewDocumentCommand \R {} {{\mathbb{R}}}
\NewDocumentCommand \N {} {{\mathbb{N}}}

\NewDocumentCommand \C {} {{\mathbb{C}}}

\let\mathccal\mathcal
\let\symcal\mathcal
\let\symbf\mathbf
\NewDocumentCommand \xpar { m } {\left( #1\right)}
\NewDocumentCommand \xcpar { m } {\left[ #1\right]}
\NewDocumentCommand \xxpar { m } {\left\llbracket #1\right\rrbracket}

\NewDocumentCommand \set { m } {\left\{{#1}\right\}}
\NewDocumentCommand \fun { m m m } {{#1}:{#2}\rightarrow {#3}}
\NewDocumentCommand \anfun { m m } {{#1}\rightarrow {#2}} %funzione anonima

\NewDocumentCommand \iset { m m } {\left\{{#1}\mathrel{}\middle\vert\mathrel{}{#2}\right\}}
\NewDocumentCommand \sdiff { m m } {{#1}\Delta {#2}}

\ProvideDocumentCommand \abs { m } {\left\lvert#1\right\rvert}

\NewDocumentCommand{\mval}{ m m }{\mathcal{A}_{#1}\xpar{#2}}
\NewDocumentCommand{\mdist}{ m m }{\mathccal{G}\xcpar{#1, #2}}

\NewDocumentCommand{\QP}{ m }{\mathbf{#1}}

\NewDocumentCommand \coint { m m } {\left[{#1},{#2}\right)}
\NewDocumentCommand \ccint { m m } {\left[{#1},{#2}\right]}
\NewDocumentCommand \rstr { m m } {\left.{#1}\right|_{#2}}
\DeclareMathOperator{\dist}{dist}
\DeclareMathOperator*{\aplimsup}{ap\,lim\,sup}

\numberwithin{equation}{section}

\theoremstyle{plain}
\newtheorem{theorem}{Theorem}[section]
\newtheorem{lemma}[theorem]{Lemma}
\newtheorem{proposition}[theorem]{Proposition}
\newtheorem{corollary}[theorem]{Corollary}

\theoremstyle{definition}
\newtheorem{definition}[theorem]{Definition}
\newtheorem{example}[theorem]{Example}

\theoremstyle{remark}
\newtheorem*{remark}{Remark}

\NewDocumentCommand{\cspace}{m m m m m}{\mathrm{Lip}^{#5}_{#3, #4}\xpar{#1, #2}}
\NewDocumentCommand{\cbund}{m m m}{\mathrm{Lip}^{#3}\xpar{#1, #2}}

\NewDocumentCommand{\densun}{m}{{#1}^1}
\NewDocumentCommand{\udif}{m m}{T_{#2}#1}
\NewDocumentCommand{\poym}{}{\mathtt{J}}

\NewDocumentCommand{\abgr}{ m }{\abs{\nabla #1}^+}
\NewDocumentCommand{\abgrm}{ m }{\abs{\nabla #1}^{\mu, +}}

\title{The Stepanov theorem for Q-valued functions}
\author{Paolo De Donato}
\address{Dipartimento di Matematica ``T. Levi-Civita'', via Trieste 63, 35121 Padova, Italy.}
\email{paolo.dedonato@unipd.it}
\thanks{The author is supported by the University of Padova, in particular through the STARS@unipd project ``QuASAR - Questions About Structure And Regularity of currents'' (\texttt{MASS\_STARS\_MUR22\_01}), and by the Gruppo Nazionale per l’Analisi
Matematica, la Probabilità e le loro Applicazioni (GNAMPA) of the Instituto
Nazionale di Alta Matematica (INdAM)}

\begin{document}

\begin{abstract}
In this work we prove the Stepanov differentiation theorem for multiple-valued functions. This theorem is proved in the wide generality of metric-space-multiple-valued functions without relying on a Lipschitz extension result. General definitions of differentiability and approximate differentiability for functions between suitable metric spaces are also introduced.
\end{abstract}

\maketitle

\tableofcontents

\showhyphens{area-minimizing}

\section{Introduction}
The theory of $Q$-valued functions has been developed by Almgren to study regularity properties of mass-minimizing currents. In particular, in \cite{Almgren2000} he generalized the concepts of differential and Sobolev spaces for $Q$-valued functions, i.e. functions with values in the space $\mval QM$ made by unordered $Q$-tuple of points in a metric space $M$, to study regularity properties of \virg{Dir-minimizing} functions, which are $Q$-valued functions that minimize a suitable generalization of the Dirichlet energy for $Q$-valued functions. His work about Dir-minimizing functions has been readapted and simplified in \cite{DeLellis2011}.

When $M=\R^k$, we say that a $Q$-valued function $\fun f{\Omega\subseteq\R^n}{\mval Q{\R^k}}$ is differentiable at $x\in\Omega$ with $f(x)=\xxpar{P_1}+\xxpar{P_2}+\dotsb+\xxpar{P_Q}$ if and only if there exist $M_1, M_2, \dotsc, M_Q$ $(k\times n)$-matrix such that, setting
\begin{gather*}
F(z)=\sum^Q_{i=1}\xxpar{P_i+M_iz}\text ,
\end{gather*}
we have
\begin{gather*}
    \lim_{y\to x}\frac{\mdist{f(y)}{F(y-x)}}{\abs{y-x}}=0
\end{gather*}
(where $\mdist{\cdot}{\cdot}$ is the canonical metric on $\mval QM$ as defined in section \ref{sec:qvl}) and if $P_i=P_j$ for some $i, j$ then $M_i=M_j$.

In \cite{DeLellis2011} the following version of the Rademacher theorem for $Q$-valued functions is proved:
\begin{theorem}
\label{rade}
If $\fun f{\R^n}{\mval Q{\R^k}}$ is a Lipschitz $Q$-valued function then $f$ is a.e.\ differentiable on $\R^n$.
\end{theorem}
To prove this, the authors used a Lipschitz extension result for $Q$-valued Lipschitz functions from a closed subset of $\R^n$ to all $\R^n$. However, we cannot expect a Lipschitz extension result when we replace $\R^k$ with a generic metric space, because it is not always possible to globally extend a continuous function defined on a closed subset of $\R^n$ with values on a disconnected metric space.

However, an extension result for $Q$-valued Lipschitz functions is not strictly required, since it is possible to prove theorem \ref{rade} without using the extension result for Lipschitz $Q$-valued functions at the cost of introducing a slightly weaker notion of differentiable function (which coincides with the classical one on interior points).

One goal of this work is to prove the Stepanov theorem for $Q$-valued functions:
\begin{theorem}\label{main}
Let $\fun f{B}{\mval Q{\R^k}}$ be a measurable function defined on a measurable set $B\subseteq\R^n$ without isolated points. Let
\begin{gather*}
\abgr f(x)=\limsup_{B\ni y\to x}\frac{\mdist{f(y)}{f(x)}}{\abs{y-x}}\\
\abs{\nabla f}^{\symcal L^n, +}(x)=\aplimsup_{B\ni y\to x}\frac{\mdist{f(y)}{f(x)}}{\abs{y-x}}\\
A_f = \iset{x\in B}{\abgr f(x)<+\infty}\\
A^{\symcal L^n}_f = \iset{x\in B}{\abs{\nabla f}^{\symcal L^n, +}(x)<+\infty}\text .
\end{gather*}

Then $f$ is a.e.\ differentiable on $A_f$ and a.e.\ approximately differentiable on $A^{\symcal L^n}_f$.
\end{theorem}

A proof of the a.e.\ differentiability on the set $A_f$ can be found in corollary 2.8 of \cite{Menne2009} (see also the proof of theorem 3.1.9 in \cite{Federer1996}). In our work, we generalize this result to functions defined on doubling metric measure spaces taking values in complete metric spaces, without requiring the functions in question to be measurable in the classical sense.

In Section \ref{sec:diff} of this work we introduce a broad definition of \virg{differential classes} for functions defined between two metric spaces. A function $\fun{f}{\R^n}{\R^k}$ is differentiable at some point $x\in \R^n$ if and only if there exists an affine function $F(y)=Ay+z_0$, where $z_0\in \R^k$ and $A$ is a $k\times n$ matrix, such that $\abs{f(y)-F(y)}$ is infinitesimal compared with $\abs{y-x}$. When $f$ is defined between Hilbert or Banach space we are still able to define a similar notion of differentiability, which is well-known under the name of Fréchet derivative.

Both these notions of differentiability rely on approximating our function $f$ with some affine function $F$ near $x$. However, we can replace $F$ with any Lipschitz function $\tilde{F}$ defined only in a neighborhood of $x$ that behaves like $F$ near $x$, that is $\lim_{y\to x}\frac{\abs{F(y)-\tilde{F}(y)}}{\abs{y-x}}=0$. All these locally defined Lipschitz functions form a class of equivalence $\xcpar{F}$ which uniquely identify our initial affine function $F$, which we will call in Section \ref{sec:diff} \emph{differential germ} induced by $F$, and the set of all these differential germs is called \emph{distribution}.

Notice that we can define the class of equivalence $\xcpar{F}$ even if $F$ is not an affine function, we only need that $F$ is Lipschitz on a neighborhood of $x_0$. Moreover, we can take any set $\symcal D$ containing some of these classes of equivalence and say $f$ is $\symcal D$-differentiable at $x_0$ if and only if there exists $\xcpar{\tilde{F}}\in\symcal D$ such that $\lim_{x\to x_0}\frac{\abs{f(x)-\tilde{F}(x)}}{\abs{x-x_0}}=0$. Each of these sets $\symcal D$, which are still called distributions, introduces a differentiation structure for functions from $\R^n$ into $\R^k$ which can be richer or poorer than the one induced by the \virg{standard distribution} generated by affine functions.

When $f$ is defined on a doubling metric measure space $\xpar{X, \poym, \mu}$ with values in a complete metric space $(M, d)$ we do not have a notion of affine function, but we can still define generic distributions as explained before. Simply, we no longer have a standard distribution as we do in Euclidean and Banach spaces. On the other hand, this general notion of differentiation through distributions can be adapted to single-valued functions defined between Euclidean spaces, $Q$-valued functions, and real-valued functions defined on $p$-Poincaré metric measure spa\-ces (see \cite{Cheeger1999, Keith2004, Heinonen2007}). We then say $\fun{f}{X}{M}$ is $\symcal D$-differentiable at $x\in X$ if and only if there exists a differential germ $\xcpar{F}\in\symcal D$ such that $\lim_{y\to x}\frac{d\xpar{f(y), F(y)}}{\poym(y, x)}=0$.

Moreover, we say $f$ is approximately differentiable at $x$ if for every $\varepsilon>0$ there exists a $\mu$-measurable set $U\subseteq X$ which has density $1$ at $x$ such that $\frac{d\xpar{f(y), F(y)}}{\poym\xpar{y, x}}<\varepsilon$ for every $y\in U\setminus\set{x}$, or equivalently $\mu-\lim_{y\to x}\frac{d\xpar{f(y), F(y)}}{\poym\xpar{y, x}}=0$ where we have used the symbol $\mu-\lim$ to denote a limit on the topology where every neighborhood of $x$ contains a $\mu$-measurable set with density $1$ at $x$.

Still in Section \ref{sec:diff} we investigate on which conditions a distribution $\symcal D$ must satisfy to get a Stepanov-like result, for both differentiability and approximate differentiability, on such metric spaces. Indeed, if $\symcal D$ has too few differential germs then some Lipschitz functions would be not differentiable on a set with measure strictly greater than zero. The main goal of this section is to prove the equivalence of the Stepanov theorem for a generic distribution $\symcal D$ with the a.e. differentiability (still for $\symcal D$) of any Lipschitz function defined on closed subsets of $X$, as stated in theorems \ref{difcom:aplim} and \ref{difcom}.

Section \ref{sec:qvl} is instead entirely devoted to differentiability properties of metric $Q$-valued functions, that is functions with values in $\mval QM$ instead of $\mval Q{\R^k}$ with $M$ any complete metric space. In this section we use the results provided in section \ref{sec:diff} to prove that if the Stepanov theorem holds for every function from $X$ to $M$ with respect to some distribution $\symcal D$ then it should hold also for functions from $X$ to $\mval{Q}{M}$ with respect to a suitable distribution $\symcal D^Q$ deduced from $\symcal D$. 

We finally use the results proved in these two sections to get the following generalized version of Stepanov's theorem for $Q$-valued functions, which in turn generalizes theorem \ref{main}:

\begin{theorem}\label{maingen}
Let $(X, \poym, \mu)$ be a doubling metric measure space, $(M, d)$ be a complete metric space, and $\symcal D$ a complete (approximate) $X$-dis\-tri\-bu\-tion on $M$. Then every $Q$-valued function $\fun fX{\mval QM}$ is a.e.\ (approximately) $\symcal D^Q$-differentiable on $A_f$ (or $A_f^\mu$), where $A_f$ and $A^\mu_f$ are defined respectively in \eqref{af:j} and \eqref{af:mu}.

In particular, these properties hold whenever $M=\R^k$ and $X$ is $\R^n$ or a $p$-Poincaré doubling metric measure space (thanks to \cite{Cheeger1999}).
\end{theorem}

\section{Generalized differentiability}\label{sec:diff}
\subsection{Differential germs}
Let $(X, \poym, \mu)$ be a $K$-doubling metric measure space ($K>0$), that is, $(X, \poym)$ is a complete separable metric space and $\fun\mu{\symcal P(X)}{\ccint 0{+\infty}}$ is a Borel-regular outer measure on $X$ such that for every $x\in X$ and every $r>0$ we have 
\[
0<\mu\xcpar{B_{2r}(x)}\leq K\mu\xcpar{B_r(x)}<+\infty\text .
\]
\begin{remark}
It is simple to show that for every $x, y\in X$ and $r, s>0$
\begin{gather*}
\mu\xcpar{B_r(x)}\leq \symcal L\xpar{K, \frac{r+\poym(x, y)}{s}}\mu\xcpar{B_s(y)}
\end{gather*}
where $\symcal L(\alpha, \beta)>0$ is a suitable function. Indeed we have
\begin{multline*}
\mu\xcpar{B_r(x)}\leq \mu\xcpar{B_{r+\poym(x, y)}(y)}\leq\mu\xcpar{B_{2^{P\xcpar{\xpar{r+\poym(x, y)}/s}}s}(y)}\\
\leq K^{P\xcpar{\xpar{r+\poym(x, y)}/s}}\mu\xcpar{B_s(y)}
\end{multline*}
where $P(t)=\max\set{0, \lceil\log_2 t\rceil}$.
\end{remark}

\begin{proposition}
Let $E\subseteq X$ be a generic set, there exists a Borel set $F\supseteq E$ such that $\mu\xcpar{F'\cap B_r(x)}\geq \mu\xcpar{F\cap B_r(x)}$ for every $x\in X, r>0$ and every other measurable set $F'\supseteq E\cap B_r(x)$.
\end{proposition}
\begin{proof}
Fix a point $x_0\in X$ and set $B_R=B_R\xpar{x_0}$ for every $R\in\N$. By hypothesis we have $0<\mu\xpar{B_R}<+\infty$ and there exists a Borel set $F_R$ such that $E\cap B_R\subseteq F_R\subseteq B_R$ and $\mu\xpar{F_R}=\mu\xpar{E\cap B_R}$ since $\mu$ is Borel-regular. Then for every $B_r(x)\subseteq B_R$ we have
\begin{gather*}
\begin{split}
\mu\xpar{E\cap B_R}&\leq \mu\xpar{E\cap B_r(x)}+\mu\xpar{E\cap B_R\setminus B_r(x)}\\
&\leq \mu\xpar{F_R\cap B_r(x)}+\mu\xpar{F_R\setminus B_r(x)}\\
&=\mu\xpar{F_R}=\mu\xpar{E\cap B_R}
\end{split}
\end{gather*}
which implies $\mu\xpar{E\cap B_r(x)}=\mu\xpar{F_R\cap B_r(x)}$ for every $B_r(x)\subseteq B_R$.

Now let $F=\bigcup^{\infty}_{S=1}\bigcap^{\infty}_{R=S}F_R$, clearly $E\subseteq F$ and $F$ is a Borel set. For every ball $B_r(x)\subseteq X$ there exists $\tilde{S}\in\N$ such that $B_r(x)\subseteq B_{\tilde S}$ and so
\begin{gather*}
\mu\xpar{F\cap B_r(x)}=\lim_{S\to +\infty}\mu\xpar{\bigcap^{+\infty}_{R=S}\xpar{F_R\cap B_r(x)}}\leq \mu\xpar{E\cap B_r(x)}\text .
\end{gather*}
Therefore $\mu\xpar{F\cap B_r(x)}=\mu\xpar{E\cap B_r(x)}\leq \mu\xpar{F'\cap B_r(x)}$ for any other Borel set $F'$ containing $E\cap B_r(x)$.
\end{proof}

A point $x$ of $E$ is said to be a $\mu$-interior point (for $E$) if and only if there exists $B\subseteq E$ Borel set such that
\begin{gather*}
    \lim_{r\to 0^+}\frac{\mu\xcpar{B_r(x)\setminus B}}{\mu\xcpar{B_r(x)}}=0\text .
\end{gather*}
Let also $\densun E\subseteq E$ be the set of all the $\mu$-interior points for $E$, and if $x\in\densun E$ then we say that $E$ is a $\mu$-nei\-ghbor\-hood for $x$.

Notice that $E$ and $\densun E$ may not be $\mu$-measurable sets, however if $E$ is $\mu$-measurable then $\mu\xpar{E\setminus\densun E}=0$ by using the doubling property of $\mu$ (see for example \cite{Rigot2021}). Moreover, if $E$ and $F$ are two $\mu$-nei\-ghbor\-hoods for $x$ then $E\cap F$ is another $\mu$-neighborhood for $x$ because
\begin{gather*}
\lim_{r\to 0^+}\frac{\mu\xcpar{B_r(x)\setminus (E\cap F)}}{\mu\xcpar{B_r(x)}}\leq \lim_{r\to 0^+}\frac{\mu\xcpar{B_r(x)\setminus E}}{\mu\xcpar{B_r(x)}}+\lim_{r\to 0^+}\frac{\mu\xcpar{B_r(x)\setminus F}}{\mu\xcpar{B_r(x)}}=0\text .
\end{gather*}

Therefore the family of $\mu$-neighborhoods defines a topology on $X$, that we will call the $\mu$-topology. This topology is stronger than the metric topology induced by $\poym$ on $X$ (which we will also call $\poym$-topology).

\begin{proposition}
For every $x\in X$ all the following statements are equivalent:
\begin{itemize}
\item $x$ is an isolated point (for $X$) with respect to the $\poym$-topology;
\item $x$ is a $\mu$-isolated point;
\item $\mu\xpar{\set x}>0$.
\end{itemize}
\end{proposition}
\begin{proof}
If $x$ is an isolated point then there exists $r>0$ such that $B_r(x)=\set x$, and since $B_r(x)$ is itself a $\mu$-neighborhood for $x$ then this would automatically imply that $x$ is a $\mu$-isolated point.

Next if $x$ is $\mu$-isolated then $\set x$ is a $\mu$-neighborhood for $x$ so
\begin{gather*}
\lim_{r\to 0^+}\frac{\mu\xcpar{B_r(x)\setminus \set x}}{\mu\xcpar{B_r(x)}}=0
\end{gather*}
which is possible only if $\mu(\set x)>0$ otherwise $\mu\xcpar{B_r(x)\setminus \set x}=\mu\xcpar{B_r(x)}$.

Finally, we only need to prove that if $\mu(\set x)=\delta>0$ then $x$ is an isolated point, that is $B_r(x)=\set x$ for some $r>0$. First to all notice that, since $\mu\xcpar{B_r(x)}<+\infty$, $\lim_{r\to 0^+} \mu\xcpar{B_r(x)\setminus\set x}=0$. Therefore there exists $\tilde r>0$ such that $\mu\xcpar{B_r(x)\setminus\set x}<\delta/K$ for every $0<r<\tilde r$.

By contradiction $x$ is not an isolated point in $X$, therefore fixed $0<r<\tilde r$ there exist $y\neq x$ and $s>0$ such that
\begin{itemize}
\item $B_s(y)\subseteq B_r(x)\setminus\set x$;
\item $x\in B_{2s}(y)$.
\end{itemize}
Then
\begin{gather*}
\mu\xcpar{B_s(y)}\leq \mu\xcpar{B_r(x)\setminus\set x}<\frac\delta K=\frac 1K\mu\xpar{\set x}\leq \frac 1K\mu\xcpar{B_{2s}(y)}\leq \mu\xcpar{B_s(y)}
\end{gather*}
which is impossible, therefore $\set x$ must be an isolated point.
\end{proof}

\begin{proposition}
\label{urra}Consider any nonempty set $U\subseteq X$. Then for every $x\in \densun U$ we can find a map $\fun{r_x}{X}{U}$ such that $r_x(y)=y$ on $U$ and
\begin{gather*}
    \poym\xcpar{y, r_x(y)}=o\xpar{\poym\xpar{y, x}}\qquad\text{for }\poym(y, x)\to 0\text .
\end{gather*}
\end{proposition}
\begin{proof}
We can assume without loss of generality that $x$ is an accumulation point (that is, $B_r(x)\neq \set x$ for every $r>0$ or equivalently $\mu(\set x)=0$) and $U$ is a Borel set. For every $N\in\N$ we can find $\delta_N>0$ such that for every $0<r<\delta_N$
\begin{gather*}
    \frac{\mu\xcpar{B_r(x)\setminus U}}{\mu\xcpar{B_r(x)}}\leq \frac 1N\text .
\end{gather*}

In order to prove our statement we need only to prove that for every $\varepsilon>0$ there exists $\delta>0$ so that for every $y\in X$ with $\poym\xpar{y, x}<\delta$ there exists $y'\in U$ such that
\begin{gather}
    \label{disce}
    \poym\xpar{y, y'}< \varepsilon\poym\xpar{y, x}\text .
\end{gather}
If \eqref{disce} does not hold then there exists $\varepsilon_0>0$ such that for every $N\in\N$ there exists $y_N$ such that $0<r_N=\poym\xpar{y_N, x}<\delta_N/\xpar{1+\varepsilon_0}$ but $B_{\varepsilon_0r_N}\xpar{y_N}\cap U=\emptyset$. Therefore, since we have $\xpar{1+\varepsilon_0}r_N<\delta_N$,
\begin{multline*}
\mu\xcpar{B_{\varepsilon_0r_N}\xpar{y_N}}\leq \mu\xcpar{B_{\xpar{1+\varepsilon_0}r_N}(x)\setminus U}\leq\frac 1N\mu\xcpar{B_{(1+\varepsilon_0)r_N}(x)}\\
\leq \frac 1N\mu\xcpar{B_{(2+\varepsilon_0)r_N}(y_N)}\leq \frac{K_{\varepsilon_0}}{N}\mu\xcpar{B_{\varepsilon_0r_N}(y_N)}
\end{multline*}
where $K_{\varepsilon_0}$ is a constant that depends only on $\varepsilon_0>0$ and $\mu$ but not on $N$. Therefore, for $N$ sufficiently large we will get a contradiction.
\end{proof}

Limits in the $\mu$-topology coincide with the approximate limits with respect to measure $\mu$, that is let $\fun gA\R$ be a function defined on a nonempty set $A\subseteq X$ and let $x\in X$ be a $\mu$-accumulation point for $A$ (that is, $A\cap U\setminus\set x\neq\emptyset$ for every $\mu$-neighborhood $U$ of $x$). Then we have $\mu-\lim_{A\ni y\to x}g(y)=z\in\R$ if and only if for every $\varepsilon>0$ there exists a $\mu$-neighborhood $U$ of $x$ such that $\abs{g(y)-z}<\varepsilon$ for every $y\in A\cap U\setminus\set x$.

Moreover, for real-valued functions $\fun gA\R$ and $\mu$-accumulation point $x$ for $A\subseteq X$ we set 
\begin{gather*}
\mu-\limsup_{A\ni y\to x}g(y)=\inf_{x\in \densun U}\sup_{y\in U\cap A\setminus\set x}g(y)\text .
\end{gather*}

\begin{proposition}
\label{diffty}Let $\fun{f}B{\coint{0}{+\infty}}$ be a real-valued Lipschitz function defined on a $\mu$-neighborhood $B$ of $x$ with $\mu(\set x)=0$ and $f(x)=0$. Then
\begin{gather*}
\limsup_{B\ni y\to x}\frac{f(y)}{\poym(y, x)}=\mu-\limsup_{B\ni y\to x}\frac{f(y)}{\poym(y, x)}\text .
\end{gather*}
\end{proposition}
\begin{proof}
The inequality $\geq$ is trivial. To prove the opposite inequality take any $\mu$-neigh\-borhood $U$ of $x$ and set $C=U\cap B\setminus\set x$. Let $\fun{r_x}{X\setminus\set x}C$ as in proposition \ref{urra}, then for every $\varepsilon>0$ there exists $\delta>0$ such that if $y\in B_\delta(x)$ then $\poym[r_x(y), y]\leq\varepsilon\poym(y, x)$. Therefore,
\begin{multline*}
\sup_{y\in B_\delta(x)\cap B\setminus\set x}\frac{f(y)}{\poym(y, x)}\leq \sup_{y\in C}\frac{f(y)}{\poym(y, x)}\sup_{y\in B_\delta(x)\setminus\set x}\frac{\poym\xcpar{r_x(y), x}}{\poym(y, x)}\\
+ \sup_{y\in B_\delta(x)\cap B\setminus\set x}\frac{\abs{f(y)-f[r_x(y)]}}{\poym(y, x)}\leq (1+\varepsilon)\sup_{y\in C}\frac{f(y)}{\poym(y, x)}+K\varepsilon
\end{multline*}
for some $K>0$. Thus
\begin{gather*}
\limsup_{B\ni y\to x}\frac{f(y)}{\poym(y, x)}\leq (1+\varepsilon)\mu-\limsup_{B\ni y\to x}\frac{f(y)}{\poym(y, x)} +K\varepsilon
\end{gather*}
which proves our statement.
\end{proof}

Analogously to the definitions of $\abgr{f}(x)$ and $A_f$ stated in theorem \ref{main} for $Q$-valued functions, we the following quantities for a generic function $f$ defined on a nonempty subset $A$ of $X$ with values in some metric space $\xpar{M, d}$:
\begin{align*}
\abgr f(x) &= \limsup_{A\ni y\to x}\frac{d\xcpar{f(y), f(x)}}{\poym(x, y)}\\
\abgrm f(x) &= \mu-\limsup_{A\ni y\to x}\frac{d\xcpar{f(y), f(x)}}{\poym(x, y)}
\end{align*}
\begin{align}
A_f &=\iset{x\in A}{x\ \text{accumulation point, }\abgr f(x)<+\infty}\label{af:j}\\
A_f^\mu &=\iset{x\in A}{x\ \mu\text{-accumulation point, }\abgrm f(x)<+\infty}\label{af:mu}\text .
\end{align}
Clearly we have $\abgrm f(x)\leq \abgr f(x)$ and $A_f\subseteq A_f^\mu\cup N$ where $\mu(N)=0$, because $A_f\setminus A_f^\mu$ may contain accumulation points for $A$ that are $\mu$-isolated. This set should have measure zero since $\mu$ is Borel regular and so there exists $\tilde A\supseteq A$ Borel set such that $\mu\xcpar{B_r(x)\cap\tilde A}=\mu\xcpar{B_r(x)\cap A}$ for every $x\in X$, $r>0$.

Even if $f$ is a Lipschitz function defined on a closed set $C\subseteq X$ we can have $A_f^\mu\neq C$, however if $C$ only contains accumulation points for $X$ then $\mu\xpar{C\setminus A_f^\mu}=0$.

Now, consider a complete metric space $\xpar{M, d}$. We want to define a notion of differentiability for functions defined on $X$ with values in $M$. For every topology $\tau$ on $X$ stronger than or equal to the $\poym$-topology (in particular, we set $\tau=\poym$ for the $\poym$-topology and $\tau=\mu$ for the $\mu$-topology), every $x\in X$, $P\in M$ let $\cspace MXxP\tau$ be the set of all the equivalence classes $\xcpar{F}_\sim$, where $\fun FBM$ is a Lipschitz function defined on a $\tau$-nei\-ghbor\-hood $B\subseteq X$ of $x$ such that $F(x)=P$, with respect to the equivalence relation 
\[ 
F\sim G\Leftrightarrow \lim_{y\to x} \frac{d[F(y), G(y)]}{\poym\xpar{y, x}}=0\text{ (with respect to }\tau\text{ topology).}
\]

For every $(x, P)\in X\times M$ the space $\xcpar{\cspace MXxP\tau, d_{x, P}, 0_P, \abs{\cdot}_{x, P}}$ is a pointed metric space where
\begin{gather*}
d_{x, P}([F], [G])=\limsup_{y\to x}\frac{d[F(y), G(y)]}{\poym\xpar{y, x}}\\
0_P=\xcpar{\anfun{x\in X}{P\in M}}\\
\abs{[F]}_{x, P}=d_{x, P}\xpar{[F], 0_P}\text .
\end{gather*}
We also define
\begin{gather*}
\cbund MX\tau=\bigsqcup_{(x, P)\in X\times M}\cspace MXxP\tau
\end{gather*}
and we call the elements of $\cbund MX\tau$ the $X$-\emph{differential germs} of $M$ with respect to the topology $\tau$. In particular, if $\tau=\poym$ then the elements of $\cbund MX\poym$ are just called $X$-\emph{differential germs}, whereas if $\tau=\mu$ then elements of $\cbund MX\mu$ are called \emph{approximate} $X$-\emph{differential germs}.

When $\tau$ is the topology induced by the metric $\poym$ then arbitrary subsets $\symcal D$ of $\cbund MX\poym$ are called $X$\emph{-distributions} for $M$, and for every $(x, P)\in X\times M$ we set $\symcal D_{x, P}=\symcal D\cap\cspace MXxP\poym$. Analogously, when $\tau$ is the topology generated by the measure $\mu$ as before then arbitrary subsets of $\cbund MX\mu$ are instead called \emph{approximate} $X$-\emph{distributions}.

\subsection{Differentiation and completeness}
\begin{definition}
Let $A$ be a nonempty subset of $X$, $\symcal D$ be an $X$-distribution, and let $\fun fAM$ be an arbitrary function. For every $x\in A$ accumulation point we say that $f$ is $\symcal D$\emph{-differentiable} at $x$ if and only if there exists $[\fun FBM]\in\symcal D_{x, f(x)}$ such that
\begin{gather}\label{diff}
\lim_{A \ni y\to x}\frac{d\xcpar{f(y), F(y)}}{\poym\xpar{y, x}}=0\text .
\end{gather}
Moreover, the differential germ $[F]$ that satisfies \eqref{diff} is a $\symcal D$\emph{-differential} of $f$ at $x$.

If $\symcal D$ is an approximate $X$-distribution and $x$ is a $\mu$-accumulation point for $A$ then $f$ is \emph{approximately $\symcal D$-differentiable} at $x$ if and only if
\begin{gather*}
\mu-\lim_{A \ni y\to x}\frac{d\xcpar{f(y), F(y)}}{\poym\xpar{y, x}}=0
\end{gather*}
and in that case $[F]$ is an \emph{approximate $\symcal D$-differential} for $f$ at $x$.
\end{definition}

\begin{remark}
If $x$ is a $\mu$-interior point for $A$ then $f$ admits at most one approximate differential at $x$. However, if $x$ is a boundary point then the differential may not be unique and $f$ could have infinitely many different differentials.
\end{remark}

\begin{proposition}
\label{uniq1}Let $\fun fAM$ be an arbitrary function that is $\symcal D$-dif\-fe\-ren\-tia\-ble at some accumulation point $x\in A$. Then $x\in A_f$ and $\abgr f(x)\leq \abs{F}_{x, f(x)}$ where $[F]\in\symcal D_{x, f(x)}$ is any differential for $f$. 

If in addition $A$ is a $\mu$-neighborhood for $x$ then $\abgr f(x)=\abs{\udif fx}_{x, f(x)}$ where $\udif fx=\xcpar{F}\in\symcal D_{x, f(x)}$ is the unique differential of $f$ at $x$.
\end{proposition}
\begin{proof}
Let $[F]\in\symcal D_{x, f(x)}$ be a $\symcal D$-differential of $f$ at $x$, then
\begin{multline*}
\abgr f(x)=\limsup_{A\ni y\to x}\frac{d\xcpar{f(y), f(x)}}{\poym\xpar{y, x}}\\
\leq \limsup_{A\ni y\to x}\frac{d\xcpar{f(y), F(y)}}{\poym\xpar{y, x}}+\limsup_{y\to x}\frac{d\xcpar{F(y), F(x)}}{\poym\xpar{y, x}}=\abs{F}_{x, f(x)}\text .
\end{multline*}

Next, we prove that if $x$ is a $\mu$-interior point for $A$ then $f$ has exactly one $\symcal D$-differential at $x$ and equality holds in the last inequality. Let $[G]\in\symcal D_{x, f(x)}$ be another $\symcal D$-differential of $f$ at $x$, then we immediately get
\begin{gather*}
\lim_{A\ni y\to x}\frac{d\xcpar{F(y), G(y)}}{\poym(y, x)}=0
\end{gather*}
but this does not yet imply that $[F]=[G]$. Since $A$ is a $\mu$-neighborhood of $x$ we immediately get
\begin{multline*}
\mu-\limsup_{y\to x}\frac{d\xcpar{F(y), G(y)}}{\poym(y, x)}=\mu-\lim_{A\ni y\to x}\frac{d\xcpar{F(y), G(y)}}{\poym(y, x)}\\
\leq \lim_{A\ni y\to x}\frac{d\xcpar{F(y), G(y)}}{\poym(y, x)}=0\text .
\end{multline*}
Now both $F$ and $G$ are defined on a neighborhood of $x$, by proposition \ref{diffty} we finally get
\begin{gather*}
\limsup_{y\to x}\frac{d\xcpar{F(y), G(y)}}{\poym(y, x)}=\mu-\limsup_{y\to x}\frac{d\xcpar{F(y), G(y)}}{\poym(y, x)}=0
\end{gather*}
thus $[F]=[G]$.

Assime again that $A$ is a $\mu$-neighborhood for $x$, and let $\udif fx=[F]$ be the unique $\symcal D$-differential of $f$ at $x$. Then
\begin{multline*}
\abs{\udif fx}_{x, f(x)}=\mu-\limsup_{y\to x}\frac{d[F(y), F(x)]}{\poym(y, x)}\\
=\mu-\limsup_{A\ni y\to x}\frac{d[F(y), F(x)]}{\poym(y, x)}\leq\abgr f(x)\text .\qedhere
\end{multline*}
\end{proof}

Similarly, we can prove the following:

\begin{proposition}
\label{uniq2}Let $\fun fAM$ be an arbitrary function that is approximately $\symcal D$-differentiable at some $x\in A$. Then $x\in A_f^\mu$ and $\abgrm f(x)\leq \abs{F}_{x, f(x)}$ where $[F]\in\symcal D_{x, f(x)}$ is any approximate differential for $f$.

If in addition $A$ is a $\mu$-neighborhood for $x$ then the approximate $\symcal D$-differential $\udif fx=[F]$ of $f$ at $x$ is unique with $\abgrm f(x)=\abs{\udif fx}_{x, f(x)}$.
\end{proposition}

We then say that a function $\fun fAM$ defined on a nonempty set $A\subseteq X$ is \emph{a.e.\ $\symcal D$-differentiable} on $A_f$ if and only if on every $\mu$-measurable subset $B$ of $A_f$ we have that $f$ is $\symcal D$-dif\-feren\-tiable for a.e.\ point of $B$. In the same way, we define functions a.e.\ approximately $\symcal D$-dif\-feren\-tiable on $A_f^\mu$.

\begin{definition}
An $X$-distribution/approximate $X$-distribution $\symcal D$ of $M$ is \emph{differentiably complete} if it satisfies the following \virg{Rademacher's condition}:
\begin{description}
    \item[(RC)] for every closed nonempty subset $C\subseteq X$ (with respect to the $\poym$-topology) and every Lipschitz function $\fun gCM$ we have that $g$ is a.e.\ approximately $\symcal D$-differentiable on $A_g^\mu$. 
\end{description}
\end{definition}

\begin{remark}
For Lipschitz functions defined on closed subsets $\symcal D$-differentia\-bility and approximate $\symcal D$-dif\-feren\-tiabi\-lity are equivalent for a.e.\ point of $C$, where $\symcal D$ is an $X$-dis\-tri\-bu\-tion. Moreover, both $A_g$ and $A_g^\mu$ are Borel sets with $\mu\xpar{\sdiff{A_g}{A_g^{\mu}}}$ $=0$.
\end{remark}

\begin{remark}
In (RC) we cannot assume $C=X$ since for a generic metric space $M$ we do not have a local Lipschitz extension result.
\end{remark}

It is trivial to show that
\begin{enumerate}
    \item $\symcal D=\cbund MX\mu$ is differentiably complete;
    \item if $\symcal D$ is differentiably complete then every $\symcal C\supseteq\symcal D$ is also differentiably complete;
    \item  if $\symcal D^1, \symcal D^2, \symcal D^3, \dotsc$ are all differentiably complete then $\symcal D=\bigcap^\infty_{i=1}\symcal D^i$ is differentiably complete. 
\end{enumerate}

Of these statements, only (3) is not trivial, however its proof is quite simple thanks to uniqueness results contained in propositions \ref{uniq1} and \ref{uniq2}. Let then $\fun fCM$ be a Lipschitz function defined on a closed subset $C$ of $X$, then there exists a measurable subset $V=\densun V$ of $C$ such that for every $i\in\N$ $f$ is approximately $\symcal D^i$-differentiable at $x$ with differential $\xcpar{F_i}\in\symcal D^i$. By uniqueness we have $\xcpar{F_i}=\xcpar{F_1}=\udif fx$ for every $i\in\N$, therefore $\udif fx\in\bigcap^\infty_{i=1}\symcal D^i_{x, f(x)}=\symcal D_{x, f(x)}$.

\begin{example}
Let $\set{X_\alpha}_{\alpha\in\N}$ be a disjoint Borel cover of $(X, \poym, \mu)$ and let $\fun{\Phi_\alpha}{X_\alpha}{\R^{N_\alpha}}$ be a sequence of Lipschitz functions with $N_\alpha\in\N$.

We also assume that $\mu\xpar{X_\alpha}>0$ and that $X_\alpha$ does not contain any $\mu$-isolated point (with respect to itself). Therefore for every $k\in\N$, every $x\in \densun{\xpar{X_\alpha}}$ and every $p\in\R^k$ we can set
\begin{gather*}
\symcal D_{x, p}=\iset{\xcpar{\anfun{y\in X_\alpha}{p+L\xcpar{\Phi_\alpha(y)-\Phi_\alpha(x)}}}}{\fun L{\R^{N_\alpha}}{\R^k}\text{ linear}}\text .
\end{gather*}
Clearly $\symcal D=\bigsqcup_{(x, p)\in X\times\R^k}\symcal D_{x, p}$ is an approximate $X$-distribution for $\R^k$ and an $X_\alpha$-distribution for $\R^k$ for every $\alpha$. 

A well-known result proven for the first time by Cheeger (\cite{Cheeger1999}) states that under some conditions on the base space $X$ (in particular, the existence of a $p$-Poincaré type inequality for some $1\leq p<+\infty$) then $\symcal D$ is a complete distribution for every $k\in\N$ and $\sup N_\alpha<+\infty$, in other words the metric space $X$ is \virg{locally finite dimensional}.
\end{example}

In the last part of this section, we will prove the equivalence of (RC) with a $\symcal D$-version of the classical Stepanov theorem for differentiability and approximate differentiability. We first need to generalize proposition \ref{urra}.

Let $\delta, \eta>0$, we say a measurable set $B\subseteq X$ containing $y$ is $(\delta, \eta)$-full at $y$ if
\begin{gather*}
    \sup_{0<r<\eta}\frac{\mu\xcpar{B_r(y)\setminus B}}{\mu\xcpar{B_r(y)}}\leq\delta\text .
\end{gather*}

\begin{lemma}\label{fulli}
There exists $\delta=\delta(K)>0$ such that the following statements hold:
\begin{enumerate}
\item \label{fulli:p1} if $B_x$ and $B_y$ are $(\delta, \eta)$-full for some $\eta>0$ at $x$ and $y$ respectively and $0<t=\poym(x, y)<\eta/3$ then $\mu\xcpar{B_{3t}(x)\cap B_x\cap B_{3t}(y)\cap B_y}>0$;
\item \label{fulli:p2} let $E\subseteq X$ be a nonempty set such that $x\in\densun{\xpar{\overline E}}$ and for every $y\in E$ let $B_y$ be a measurable set so that for every $z\in\overline E\setminus\set x$ sufficiently close to $x$ there exist $s, \eta>0$ such that $B_y$ is $(\delta, \eta)$-full at $y$ for every $y\in B_s(z)$, then there exist a $\mu$-neighborhood $U$ of $x$ and a map $\fun{r_x}UE$ such that $y\in B_{r_x(y)}$ for every $y\in U\setminus \set x$.
\end{enumerate}
\end{lemma}
\begin{proof}
Fix $\delta>0$ and assume by contradiction that (\ref{fulli:p1}) does not hold. Thus $\mu\xpar{B_x\cap B_t(x)}\leq\mu\xpar{ B_{3t}(y)\setminus B_y}$ with $3t<\eta$ and then
\begin{gather*}
\begin{split}
(1-\delta)\mu\xpar{B_t(x)}&\leq\mu\xpar{B_t(x)\cap B_x}\\
&\leq \mu\xpar{B_{3t}(y)\setminus B_y}\\
&\leq \delta\mu\xpar{B_{3t}(y)}\\
&\leq \delta\symcal L(K, 4)\mu\xpar{B_t(x)}
\end{split}
\end{gather*}
which is impossible when $\delta<1/\xcpar{1+\symcal L(K, 4)}$. 

To prove (\ref{fulli:p2}) we can assume without loss of generality that $x$ is a $\mu$-accumulation point, that is $\mu(\set x)=0$. Notice also that we can find a sequence $y_n\in E$ such that $\overline E=\overline{\iset{y_n}{n\in\N}}$. Then we need only to prove that
\begin{gather*}
\mu\xcpar{\overline E\cap B_\xi(x)\setminus \xpar{\set x\cup \bigcup_{n\in\N}B_{y_n}}}=0
\end{gather*}
for $\xi>0$ sufficiently small.

Otherwise we can find $Z\subseteq \overline E$ sufficiently near to $x$ such that $Z\cap B_{y_n}=\emptyset$ and
\begin{gather*}
    \sup_{0<r<\eta_0}\frac{\mu\xcpar{B_{r}(z)\cap Z}}{\mu\xcpar{B_{r}(z)}}\geq 1-\delta
\end{gather*}
for some $z\in Z\setminus \set x$ and $0<\eta_0$, in particular $Z$ is $(\delta, \eta_0)$-full at $z$. By assumption we can find $s>0$ and $0<\eta<\eta_0$ such that $B_{y_n}$ is $(\delta, \eta)$-full at $y_n$ for every $y_n\in B_s(z)$. Moreover, for $n$ large we also have $\poym\xpar{y_n, z}<\eta/3$ therefore by (\ref{fulli:p1}) we would have $Z\cap B_{y_n}\neq\emptyset$, a contradiction.
\end{proof}

\begin{theorem}\label{difcom:aplim}
Let $\symcal D$ be a differentiably complete approximate $X$-dis\-tri\-bu\-tion for $M$, then every function $\fun fAM$ is a.e.\ approximately $\symcal D$-dif\-fe\-ren\-tia\-ble on $A_f^\mu$.

Conversely let $\symcal D$ be an approximate $X$-distribution such that every function $\fun fXM$ is a.e.\ approximately $\symcal D$-dif\-feren\-tiable on $A_f^\mu$, then $\symcal D$ is differentiably complete.
\end{theorem}
\begin{proof}
We first prove that if $\symcal D$ is complete then $\fun fAM$ is a.e.\ approximately $\symcal D$-differentiable on every measurable subset $B$ of $A_f^\mu$. Fix a point $P\in M$ and for every $i, j\in\N$ let $E_{ij}\subseteq A^\mu_f$ be such that $x\in E_{ij}$ if and only if
\begin{enumerate}
    \item $d[f(x), P]<\frac{i}{2j}$;
    \item there exists a $(\delta, 3/j)$-full set $B_x\subseteq A$ at $x$ such that $d[f(y), f(x)]$ $\leq i\poym(y, x)$ for every $y\in B_x$.
\end{enumerate}
The main difficulty here is that the sets $E_{ij}$ are not in general measurable. By definition $A_f^\mu=\bigcup_{ij}E_{ij}$, we now prove that $\fun{f_{ij}}{x\in E_{ij}}{f(x)\in M}$ is $6i$-Lipschitz.

Take any $y, z\in E_{ij}$, if $\poym\xpar{y, z}\geq 1/j$ then
\begin{gather*}
d[f(y), f(z)]\leq d[f(y), P]+d[P, f(z)]\leq \frac ij \leq i\poym\xpar{y, z}\text .
\end{gather*}
If instead $0<t=\poym(y, z)<1/j$ then by lemma \ref{fulli} there exists $s\in B_z\cap B_y\cap B_{3t}(z)\cap B_{3t}(y)$ such that
\begin{multline*}
d[f(y), f(z)]\leq d[f(y), f(s)]+d[f(s), f(z)]\\
\leq i\poym(y, s)+i\poym(s, z)<6i\poym(y, z)\text .
\end{multline*}

Since $M$ is complete we can uniquely extend $f_{ij}$ to a Lipschitz function on $\overline{E_{ij}}$, and since $\symcal D$ is complete there exists a measurable subset $\tilde G_{ij}$ of $A_{f_{ij}}^\mu\subseteq\overline{E_{ij}}$ such that $\mu\xpar{A_{f_{ij}}^\mu\setminus \tilde G_{ij}}=0$ and $f_{ij}$ is approximately $\symcal D$-dif\-fe\-ren\-tiable on each point of $\tilde G_{ij}$. Since $E_{ij}\subseteq A^\mu_f$, every point of $\overline{E_{ij}}$ is an accumulation point for $X$, thus $\mu\xpar{\overline{E_{ij}}\setminus A_{f_{ij}}^\mu}=0$.

Let $B_{ij}=\overline{E_{ij}}\cap B$ and $G_{ij}=\densun{\xpar{\tilde G_{ij}\cap B}}$, then we immediately get 
\[
\mu\xpar{B_{ij}\setminus G_{ij}}\leq \mu\xpar{\overline{E_{ij}}\setminus \tilde G_{ij}}+\mu\xpar{\tilde G_{ij}\setminus\densun{\tilde G_{ij}}}=0
\]
and $f_{ij}$ is approximately $\symcal D$-differentiable on every point of $G_{ij}$. Let then $x\in G_{ij}$ and $\xcpar{F_{ij}}\in \symcal D_{x, f(x)}$ be the differential of $f_{ij}$ in $x$, then for every $\varepsilon>0$ there exists a measurable $\mu$-neighborhood $U_x\subseteq X$ of $x$ such that $d\xcpar{f_{ij}(y), F_{ij}(y)}<\varepsilon\poym(y, x)$ for every $y\in U_x\cap \overline{E_{ij}}$.

Let $\fun sU{E_{ij}}$ be any function defined on a $\mu$-neighborhood $U$ of $x$ such that $\poym\xcpar{s(y), x}\leq L\poym(y, x)$ near $x$ and $\poym\xcpar{s(y), y}=o\xpar{\poym(y, x)}$ for $\poym(y, x)\to 0$. Therefore,
\begin{gather*}
\begin{split}
d\xcpar{f(y), F_{ij}(y)}&\leq d\xcpar{f(y), f\xcpar{s(y)}}+d\xcpar{f_{ij}\xcpar{s(y)}, F_{ij}\xcpar{s(y)}}\\
&\qquad+d\xcpar{F_{ij}\xpar{y}, F_{ij}\xcpar{s(y)}}\\
&\leq d\xcpar{f(y), f\xcpar{s(y)}}+o\xpar{\poym\xcpar{s(y), x}}+K\poym\xpar{s(y), y}\\
&=d\xcpar{f(y), f\xcpar{s(y)}}+o\xpar{\poym\xpar{y, x}}\text .
\end{split}
\end{gather*}

We now use lemma \ref{fulli} to find an appropriate function $s$. Let for every $y\in E_{ij}\cap B_{1/j}(x)$
\begin{gather*}
    R(y)=\min\set{\frac 13, \poym(y, x)}\\
    \tilde B_y=B_y\cap B_{R(y)\poym(y, x)}(y)\text .
\end{gather*}
Clearly each $\tilde B_y$ is $\xcpar{\delta, R(y)\poym(y, x)}$-full at $y$ (since $\poym(y, x)<1/j$) and $R(y)$ depends only on $\poym(y, x)$, therefore assumptions at point (\ref{fulli:p2}) in lemma \ref{fulli} are satisfied and there exists $\fun sU{E_{ij}}$ such that $y\in \tilde B_{s(y)}$ for $y$ near $x$.

In particular, $y\in B_{s(y)}$ and $\poym(y, s(y))\leq R\xpar{s(y)}\poym(y, x)<\frac{1}{2}\poym(y, x)$ which leads to $\poym(s(y), x)<\frac{3}{2}\poym(y, x)$. Moreover, $\poym(y, s(y))\leq \poym(y, x)^2=o\xpar{\poym(y, x)}$ when $\poym(y, x)\to 0$, therefore 
\begin{gather*}
d[f(y), f[s(y)]]\leq i\poym\xcpar{y, s(y)}=o(\poym(y, x))
\end{gather*}
because $s(y)\in E_{ij}$ and $y\in B_{s(y)}$, thus $f$ is approximately $\symcal D$-dif\-fe\-ren\-tia\-ble at $x$.

The converse is trivial. Indeed if every function $\fun fXM$ is approximately $\symcal D$-differentiable a.e.\ on $A_f^\mu$ then to prove that $\symcal D$ is complete we need only to accordingly extend any Lipschitz function $\fun gCM$ on all $X$ with $C$ closed. If $\tilde g$ is any extension of $g$ on all $X$ then we always have
\begin{gather*}
\densun C\cap A^\mu_g\subseteq A_{\tilde g}^\mu
\end{gather*}
therefore $g$ is automatically a.e.\ approximately $\symcal D$-differentiable, then $\symcal D$ is complete.
\end{proof}

\begin{theorem}\label{difcom}
Let $\symcal D$ be a differentiably complete $X$-distribution for $M$, then every function $\fun fAM$ is a.e.\ $\symcal D$-differentiable on $A_f$.

Conversely let $\symcal D$ be an $X$-distribution such that every function $\fun fXM$ is a.e.\ $\symcal D$-differentiable on $A_f$, then $\symcal D$ is differentiably complete.
\end{theorem}
\begin{proof}
The first part of this result can be proved as in the previous theorem by setting $B_y=B_{3/j}(y)$ when $y\in E_{ij}$. Moreover, in the last passage we do not need lemma \ref{fulli}  to find $s$ but we can use directly proposition \ref{urra} such that $s$ is now defined in a neighborhood of $x$ instead of a $\mu$-neighborhood.

Now assume that every function $\fun fXM$ is a.e.\ $\symcal D$-differentiable on $A_f$, we will prove that $\symcal D$ is complete or equivalently that every Lipschitz function $\fun gCM$ defined on any closed subset $C\subseteq X$ is a.e.\ $\symcal D$-differentiable on $C$. We will find an extension $\tilde{g}$ of $g$ to whole $X$, and even though we cannot expect that such extension is Lipschitz it has still some useful properties. More precisely, for every map $\fun{g}{C}{M}$ defined on a closed set $C\subseteq X$ we can always find an extension $\fun{\tilde{g}}{X}{M}$ such that $A_g\subseteq A_{\tilde{g}}$.

We define $\tilde g$ in the following way: if $x\in C$ then we set $\tilde g(x)=g(x)$, instead if $x\notin C$ then we set $\tilde g(x)=g(y)$ where $y\in C$ is any point that satisfies
\begin{gather*}
    \poym\xpar{x, y}<2\dist(x, C)\text .
\end{gather*}

Let $x\in A_g\cap\partial C$, therefore there exist $K, \delta>0$ such that if $y\in C\cap B_\delta(x)$ then $d\xcpar{\tilde g(x), \tilde g(y)}=d\xcpar{g(x), g(y)}\leq K\poym\xpar{x, y}$.

Let $z\in B_{\delta/3}(x)\setminus C$, we can then find $y\in C\cap B_\delta(x)$ such that $\poym\xpar{y, z}<2\dist(z, C)$ and $\tilde g(z)=g(y)$. Therefore,
\begin{multline*}
d[\tilde g(x), \tilde g(z)]=d[g(x), g(y)]\leq K\poym\xpar{x, y}\leq K\xpar{\poym\xpar{x, z}+2\dist(z, C)}\\
\leq 3K\poym\xpar{x, z}
\end{multline*}
which implies $x\in A_{\tilde g}$ with $\abgr{\tilde g}(x)\leq 3\abgr g(x)$. Since $A_g\cap\overset{\circ}{C}=A_{\tilde{g}}\cap\overset{\circ}{C}$, where $\overset{\circ}{C}$ is the set of interior points of $C$, we have proved that $A_g\subseteq A_{\tilde{g}}$.

Now assume $\fun gCM$ is Lipschitz, then $A_g$ is a measurable subset of $C$ which is contained in $A_{\tilde{g}}$ where $\fun{\tilde g}XM$ is the extension found before. By assumption $\tilde g$ is a.e.\ $\symcal D$-differentiable on $A_{\tilde g}$, then $g$ is a.e.\ $\symcal D$-differentiable on $A_g$ and so $\symcal D$ is complete.
\end{proof}

\section{Q-valued functions}
\label{sec:qvl}
On a generic metric space $(M, d)$ we can define the space of $Q$-points $\mval Q{M}$, where $Q$ is a positive integer, as the set of all measures $\QP P$ on $M$ that can be represented as the sum of $Q$ Dirac measures $\xxpar{P_1}+\xxpar{P_2}+\dotsb + \xxpar{P_Q}$ not necessarily distinct. Then a $Q$-valued function from a set $A$ to $M$ is just a function defined on $A$ with values on $\mval QM$. Clearly, $1$-valued functions from $A$ to $M$ are equivalent to classical (single-valued) functions from $A$ to $M$.

Moreover, $Q$ single valued functions $f_1, f_2, \dotsc, f_Q$ can be used to compose a single $Q$-valued function $f(x)=\xxpar{f_1(x)}+\xxpar{f_2(x)}+\dotsb+\xxpar{f_Q(x)}$, conversely every $Q$-valued function $f$ can be globally decomposed into $Q$ single valued functions by choosing a total ordering on $M$. However, such decomposition is never useful except when $M$ is a really good space (usually $M=\R$) because many useful properties of $f$ will not be inherited by such general decompositions.

On $\mval QM$ we define the following \emph{Wassenstein metric}:
\begin{gather*}
\mdist{\QP A}{\QP B}=\mdist{\sum^Q_{i=1}\xxpar{A_i}}{\sum^Q_{i=1}\xxpar{B_i}}=\min_{\sigma\in\mathccal S_Q}\sqrt{\sum^Q_{i=1}d\xpar{A_i, B_{\sigma(i)}}^2}
\end{gather*}
where $\mathccal S_Q$ is the set of all the permutations on $\set{1, 2, \dotsc, Q}$. Therefore a $Q$-valued function $f$ from a topological/metric/measurable space $\Omega$ to $M$ is \emph{continuous/Lipschitz/measurable} if and only if $\fun f\Omega{\mval QM}$ is continuous/Lipschitz/measurable. Every measurable $Q$-valued function can be globally decomposed into $Q$ measurable single-valued functions (see \cite{DeLellis2011}), however continuous $Q$-valued functions cannot usually be decomposed into $Q$ continuous single-valued functions, for example the following function
\begin{gather*}
    \anfun{z\in\C}{\xxpar{z^{3/2}}+\xxpar{-z^{3/2}}\in\mval 2\C}
\end{gather*}
cannot be decomposed into two continuous functions in any neighborhood of the origin.

If $\symcal D$ is an $X$-distribution for a complete metric space $M$ then we are able to derive another $X$-distribution $\symcal D^Q$ but on $\mval QM$: for every $\QP P=\sum^Q_{i=1}\xxpar{P_i}$ we say $\xcpar{\symbf F}\in \symcal D^Q_{x_0, \QP P}$ if and only if there exist $\xcpar{F_i}\in\symcal D_{x_0, P_i}$ for every $i=1, 2, \dotsc, Q$ such that
\begin{gather}
    \symbf F(x)=\sum^Q_{i=1}\xxpar{F_i(x)}\ \forall x\text{ in a Borel neighborhood of }x_0\label{difdec}\\
    \text{if }P_i=P_j\text{ then }\xcpar{F_i}=\xcpar{F_j}\label{eqty}\text .
\end{gather}

Notice that $\mval QM$ is always a complete metric space whenever $M$ is complete, and can be proved by induction on $Q$. This is trivial for $Q=1$, thus assume both $M$ and $\mval QM$ are complete. Let $\QP{P}_l\in\mval {Q+1}M$ be a Cauchy sequence, by induction on $l$ we can easily find a subsequence in the form $\QP{P}_{h_l}=\xxpar{P_l}+\QP{P}'_l$, where both $P_l\in M$ and $\QP{P}'_l\in\mval QM$ are Cauchy sequences that converge to some $P\in M$ and $\QP{P}'\in\mval QM$ respectively. Therefore $\QP{P}_{h_l}\to \xxpar P+\QP{P}'$ which in turn implies $\QP{P}_{l}\to \xxpar P+\QP{P}'$ because $\QP{P}_{l}$ is Cauchy.

Many authors (see for example \cite{DeLellis2011}) impose condition \eqref{eqty} in their definition of differentiable $Q$-valued functions since it makes it easier to work componentwise. For example if $M=\R^n$ and $\xcpar{\symbf F}\in\symcal D^Q$ is a differential germ for $f=\sum^Q_{i=1}\xxpar{f_i}$ at some $x\in\R^k$ with $f_i$ continuous at $x$ then we can always take the $i$-th component $F_i\in\symcal D$ of $\symbf F$ so that $\xcpar{F_i}$ is a $\symcal D$-differential for $f_i$ at $x$.

Instead other authors (for example \cite{Simon2016}) do not use \eqref{eqty} in their definition of differentiable $Q$-valued functions, since with \eqref{eqty} we could have a sequence of $C^1$ $Q$-valued functions $f_n$ such that both $f_n$ and $\nabla f_n$ uniformly converge to some functions $f$ and $F$ respectively but $f$ is not differentiable everywhere. Therefore when one needs to work with sequences of $C^1$ $Q$-valued functions it may be useful to require only \eqref{difdec} without \eqref{eqty}.

However, if one needs only a.e.\ differentiability for functions from $\R^n$ into $\mval{Q}{\R^k}$ with respect to the Lebesgue measure on $\R^n$ (for example when dealing with Sobolev functions) then \eqref{eqty} is unimportant because the set of differentiability points for which \eqref{eqty} is not satisfied has zero measure. In this place we will always assume that germs in $\symcal D^Q$ satisfy both \eqref{difdec} and \eqref{eqty}.

\begin{theorem}
For every $\QP P\in\mval QM\setminus\iset{Q\xxpar P}{P\in M}$ there exist $1\leq R<Q$, an open neighborhood $U$ of $\QP P$ and two continuous maps $\fun{\pi_1^{\QP P}}{U}{\mval RM}$, $\fun{\pi_2^{\QP P}}U{\mval{Q-R}M}$ such that
\begin{itemize}
\item $\QP Q=\pi_1^{\QP P}(\QP Q)+\pi_2^{\QP P}(\QP Q)$ for every\/ $\QP Q\in U$;
\item $\mdist{\QP Q}{\QP R}^2=\mdist{\pi_1^{\QP P}(\QP Q)}{\pi_1^{\QP P}(\QP R)}^2+\mdist{\pi_2^{\QP P}(\QP Q)}{\pi_2^{\QP P}(\QP R)}^2$ for every\/ $\QP Q, \QP R\in U$.
\end{itemize}
\end{theorem}
\begin{proof}
Clearly there exist $1\leq R < Q$, $P$, $P_{R+1}, P_{R+2}, \dotsc, P_Q\in M$, and $\varepsilon>0$ such that $\QP P=R\xxpar P+\sum^Q_{i=R+1}\xxpar{P_i}$ and $d\xpar{P, P_j}\geq \varepsilon$ for every $R+1\leq j\leq Q$. Therefore, the subset $U\subseteq\mval QM$ defined so that $\QP A\in U$ if and only if there exists $\sigma\in S_Q$ such that $d\xpar{P, A_{\sigma(i)}}<\varepsilon/3$ for every $i=1, 2, \dotsc, R$ and $d\xpar{P_i, A_{\sigma(i)}}<\varepsilon/3$ for every $i=R+1, R+2, \dotsc, Q$ is open and contains $\QP P$. 

If $\QP A\in U$ we can assume without loss of generality that $\sigma(i)=i$, then we always have 
\begin{gather*} 
A_1, A_2, \dotsc, A_R\in B_{\varepsilon/3}(P)\\
A_{R+1}, A_{R+2}, \dotsc, A_Q\notin B_{2\varepsilon/3}(P)\\
\mdist{\QP P}{\QP A}^2=\sum^R_{i=1}d\xpar{P, A_i}^2+\sum^Q_{i=R+1}d\xpar{P_i, A_{\eta(i)}}^2
\end{gather*}
where $\eta\in S_Q$ with $\eta(i)=i$ for every $1\leq i\leq R$. We will prove only the last assertion. First of all, assume that 
\[
\mdist{\QP P}{\QP A}^2=\sum^R_{i=1}d\xpar{P, A_{\eta(i)}}^2+\sum^Q_{i=R+1}d\xpar{P_i, A_{\eta(i)}}^2
\]
and set
\begin{gather*}
E=\iset{i=1, 2, \dotsc, R}{\eta(i)\geq R+1}\\
F=\iset{i=R+1, R+2, \dotsc, Q}{\eta(i)\leq R}\text .
\end{gather*}
Set $P_i=P$ for $i=1, \dotsc, R$ we clearly have $d\xpar{P_i, A_{\eta(i)}}\geq 2\varepsilon/3$ if $i\in E\cup F$. Moreover, both $E$ and $F$ must have the same cardinality thus if $E$ is not empty then we can find another $\eta'\in S_Q$ such that $\eta'(i)=\eta(i)$ if $i\notin E\cup F$ and $d\xpar{P_i, A_{\eta'(i)}}<\varepsilon/3$ for every $i$. But this is impossible by our assumption on $\eta$.
   
Therefore we can set
    \begin{gather*}
        \fun{\pi_1^{\QP P}}{\QP A\in U}{\sum^R_{i=1}\xxpar{A_i}\in\mval RM}\\
        \fun{\pi_2^{\QP P}}{\QP A\in U}{\sum^Q_{i=R+1}\xxpar{A_i}\in\mval {Q-R}M}\text .\qedhere
    \end{gather*}
\end{proof}

\begin{corollary}
\label{sep}Let $\Omega\subseteq X$ be open nonempty and $\fun f\Omega{\mval QM}$ be continuous at $x\in\Omega\setminus\Lambda$ where
\begin{gather*}
\Lambda=\iset{x\in\Omega}{f(x)=Q\xxpar P\text{ for some }P\in M}\text .   
\end{gather*}
Then there exist an open neighborhood $U$ of $x$ and $\fun gU{\mval RM}$, $\fun hU{\mval SM}$ with $Q=R+S$, $\rstr fU=\xxpar g + \xxpar h$ and 
\begin{gather*}
\abgr g(x), \abgr h(x)\leq\abgr f(x)\text .
\end{gather*}
\end{corollary}

\begin{theorem}\label{steqv}
If an approximate $X$-distribution $\symcal D$ for $M$ is complete then the $X$-distribution $\symcal D^Q$ for $\mval QM$ satisfying both \eqref{difdec} and \eqref{eqty} is complete too. 
\end{theorem}
\begin{proof}
We use induction on $Q$. If $Q=1$ then we can use assumption (RC) on $M$, thus we can assume that theorem \ref{steqv} holds for any $1\leq Q' < Q$ in order to prove it for $Q$. Let $\fun fC{\mval QM}$ be a Lipschitz function defined on a closed subset $C$ of $X$ and let $x\in C$. 

If $x\notin\Lambda$ then by corollary \ref{sep} we have $\rstr fV=\xxpar g+\xxpar h$ with $x\in A_g\cap A_h=V$ for some open neighborhood $V$ of $x$, thus we can use induction on $g$ and $h$ to prove that $f$ is a.e.\ differentiable on $V$. Since $C$ is separable we can cover $C\setminus\Lambda$ by countably many of such $V$, thus we immediately get that $f$ is a.e.\ differentiable on $C\setminus\Lambda$.

On $\Lambda$ we have instead $f(x)=Q\xxpar{f_a(x)}$ with $\fun{f_a}{\Lambda}{M}$ Lipschitz, therefore by induction hypothesis $f_a$ is a.e.\ approximately $\symcal D$-dif\-fe\-ren\-ti\-a\-ble on $A^\mu_{f_a}$. If $L\subseteq A^\mu_{f_a}$ is the set of points on which $f_a$ is approximately $\symcal D$-differentiable with differential $[F]\in \symcal D_{x, f_a(x)}$ then we can assume that $L\subseteq\densun\Lambda\cap A_f^\mu$. In this way $F$ is approximately $\symcal D^Q$-differentiable on every point of $L$ because
\begin{gather*}
    \mdist{f(y)}{Q\xxpar{F(y)}}=\sqrt Q\ d\xcpar{f_a(y), F(y)}
\end{gather*}
for every $y\in \Lambda$.
\end{proof}

With this result we are finally able to prove theorem \ref{maingen}.

\begin{proof}[Proof of theorem \ref{maingen}]
Let $\symcal D$ be a complete $X$-distribution for $(M, d)$, then by theorem \ref{steqv} the $X$-distribution $\symcal D^Q$ for $\xpar{\mval{Q}{M}, \mathcal{G}}$ is complete for every $Q\geq 1$. Then by theorem \ref{difcom} we immediately get that every function $\fun{f}{X}{\mval{Q}{M}}$ is a.e.\ $\symcal D^Q$-differentiable on $A_f$.
\end{proof}

%\section*{Acknowledgements}

%The author is supported by the University of Padova, in particular through the\\
%STARS@unipd project ``QuASAR - Questions About Structure And Regularity of currents'' (\verb|MASS_STARS_MUR22_01|).

\printbibliography
\end{document}